\documentclass[12pt,english]{amsart}

\usepackage[paperwidth=210mm,paperheight=297mm,inner=3cm,outer=3cm,top=2.5cm,bottom=2.6cm]{geometry}

\usepackage[utf8]{inputenc} 
\usepackage[T1]{fontenc}
\usepackage{amssymb}
\usepackage[leqno]{amsmath}
\usepackage[pdftex]{graphicx}    
\usepackage{amsthm}
\usepackage{babel}
\usepackage{hyperref}
\usepackage{amsrefs}
\usepackage{enumitem}
\setlist[enumerate]{itemsep=0pt}
\setlist[itemize]{itemsep=0pt}
\setlist[enumerate,1]{label=(\arabic*)}

\newcommand\R{\mathbb{R}}

\newcommand\real{{\rm Re}}
\newcommand\imag{{\rm Im}}

\def \x{{\mathbf{x}}}

\def \I{{\mathcal I}}
\def \V{{\mathcal V}}

\def\bmatrix#1{\left[ \begin{matrix} #1 \end{matrix} \right]}

\theoremstyle{plain}
\newtheorem{Thm}{Theorem}[section]
\newtheorem{Prop}[Thm]{Proposition}
\newtheorem{Cor}[Thm]{Corollary}
\newtheorem{Lemma}[Thm]{Lemma}

\newtheorem{Conjecture}[Thm]{Conjecture}

\newtheorem*{Thm*}{Theorem}
\newtheorem*{Prop*}{Proposition}
\newtheorem*{Cor*}{Corollary}
\newtheorem*{Lemma*}{Lemma}
\newtheorem*{Sublemma*}{Sublemma}
\newtheorem*{Conjecture*}{Conjecture}

\theoremstyle{definition}

\newtheorem{Defs}[Thm]{Definitions}
\newtheorem{Example}[Thm]{Example}

\newtheorem{Remark}[Thm]{Remark}

\newtheorem*{Def*}{Definition}
\newtheorem*{Defs*}{Definitions}
\newtheorem*{Example*}{Example}
\newtheorem*{Examples*}{Examples}
\newtheorem*{LemmaDef*}{Lemma and Definition}
\newtheorem*{Notation*}{Notation}
\newtheorem*{Problem*}{Problem}
\newtheorem*{Question*}{Question}
\newtheorem*{Remark*}{Remark}

\begin{document}
\title[Testing hyperbolicity]{Testing hyperbolicity of real polynomials}

\author{Papri Dey}
\author{Daniel Plaumann}

\maketitle

\begin{abstract}
Hyperbolic polynomials are real multivariate polynomials with only
real roots along a fixed pencil of lines. Testing whether a given
polynomial is hyperbolic is a difficult task in general. We
examine different ways of translating hyperbolicity into nonnegativity
conditions, which can then be tested via sum-of-squares relaxations.
\end{abstract}

\section*{Introduction}

A real form (i.e.~homogeneous polynomial) $F$ in $n$ variables
$x_1,\dots,x_n$ is called \emph{hyperbolic} with respect to a point
$e\in\R^n$ if $F(e)\neq 0$ and $F(te-a)$ has only real roots in
$t$ for all $a\in\R^n$. The simplest example is the determinant of
symmetric matrices, which is hyperbolic with respect to the identity
matrix, and it can be useful to think of hyperbolic polynomials as
generalizations of this determinant.

Hyperbolic polynomials originate in the theory of partial differential
equations (see for example \cite{Hormander}) but have more recently received a lot of attention
in optimization (hyperbolic programming, spectrahedra) and real
algebraic geometry (determinantal representations). They are also
closely related to real-stable polynomials which have become important
in combinatorics and theoretical computer science. Indeed, if $F$ is
irreducible of degree at least $2$, then stability is equivalent to
hyperbolicity with respect to all unit vectors $e_1,\dots,e_n$.

Testing whether a given polynomial is hyperbolic (with respect to a
fixed point $e$) is a computationally difficult task as soon as
$n\ge 3$, even though the precise complexity is only known in some
special cases (see \cite{Nikhil}). In this note, we will look at three different
approaches, all of which work by translating hyperbolicity into a
condition of nonnegativity. Positive (resp.~nonnegative) polynomials
are a staple of real algebraic geometry and, while testing
nonnegativity is at least equally hard in general, there are several
well established relaxation techniques, in particular based on sums of squares.

For $n=3$, hyperbolicity is equivalent to the existence of a definite
hermitian or real symmetric determinantal representation (by a
celebrated result due to Helton and Vinnikov \cite{HV07}). For $n\ge
4$, this is no longer true (see for example \cite{NT12}). In any case,
the problem of computing determinantal representations is interesting
in its own right, but we do not consider it here (see \cite{Hen10},
\cite{PL17}, \cite{PSV12}, \cite{papriquadratic}, \cite{papribiv},
\cite{paprimulti},  \cite{Vinnikov}, \cite{Vinnikov2}).

The first method we describe is a direct translation of the classical
real root counting result due to Hermite. This is the
simplest approach. It is certainly
well known and we keep the discussion very brief. Our second method
looks at the intersection of the real and imaginary part of a
polynomial, which can be viewed as parametrized curves in the
plane. We use resultants to describe this intersection. This resultant
factors, and we show that nonnegativity of the nonntrivial factor
characterizes hyperbolicity (Thm.~\ref{rzres}); one can also apply a
real Nullstellensatz certificate to the real and imaginary parts
directly.  Our third method relies on the fact that the set of
hyperbolic polynomials is known to be connected (even simply
connected). One can explicitly trace a path from any given polynomial
to a fixed hyperbolic polynomial and characterize hyperbolicity by
evaluating a (univariate) discriminant along that path
(Thm.~\ref{Thm:NpathHyp}).

We illustrate our methods by a number of examples. The different
translations between hyperbolicity and nonnegativity are interesting
to us in themselves. From a more practical point of view, the appeal
comes mostly from the fact that sum-of-squares relaxations are already
implemented in a number of software packages and therefore readily
available. We have not written general code for our
methods that would allow for meaningful runtime-comparisons. Rather, our results
should be seen as proof-of-concept. However, the size of the semidefinite programs involved
grows rapidly with each method. The examples we have suggest that
while the Hermite method is the most straightforward (and possibly the
best in general), the intersection method should perform better in
certain cases (like curves of low degree). The discriminant method will
in general lead to larger relaxations.

Finally, it should also be pointed out that we always test hyperbolicity of a
polynomial with respect to a fixed point $e$. This seems to be the
most important case, since the point $e$ is often in some way
distinguished. However, one might also ask how to test for
hyperbolicity with respect to \textit{any} point. Apart from
completely unspecific approaches (like quantifier elimination), it
seems entirely unclear how to test this at all and it could be an
interesting future problem, even in special cases.

\textit{Acknowledgements.} We would like to thank Amir Ali Ahmadi,
Diego Cifuentes and especially Elias Tsigaridas for helpful discussions on the
subject of this paper. Much of the work on this paper has been
supported by the National Science Foundation under Grant
No. DMS-1439786 while both authors were in residence at the Institute
for Computational and Experimental Research in Mathematics in
Providence, RI, during the Fall 2018 \textit{Nonlinear Algebra}
program. The first author also gratefully acknowledges support through the
\textit{Max Planck Institute for Mathematics in the Sciences} in
Leipzig.

\section{Preliminaries}\label{sec:Hyperbolicity}

\begin{Defs}
	A polynomial in one variable with real coefficients is called \emph{real rooted} if all its complex roots are real. Fix a point $e\in\R^{n+1}$. A form (i.e.~a homogeneous polynomial) $F\in\R[x_0,x_1,\dots,x_n]$ is called  \emph{hyperbolic with respect to $e$} if $F(e)\neq 0$ and the univariate polynomial $F(te-a)\in\R[t]$
	is real rooted for all $a\in\R^{n+1}$. It is called
        \emph{strictly hyperbolic} if the roots of $F(te-a)\in\R[t]$
        are real and distinct for all $a\in\R^{n+1}$, $a\neq 0$.
\end{Defs}

\begin{Example}
  A cubic form $F\in\R[x_0,x_1,x_2]$ in Weierstra\ss{} normal form
  \[
    F(x_0,x_1,x_2)\ = \ x_0x_2^2-H(x_0,x_1)
  \]
  with $H\in\R[x_0,x_1]$ homogeneous of degree $3$ is hyperbolic (with
  respect to some point $e=(1,r,0)$) if and only if the bivariate form
  $H$ factors into three real linear forms. It is strictly hyperbolic
  if and only if these factors are distinct. In this case, the cubic
  curve defined by $F$ in the real projective plane has two connected
  components, while if $H$ contains an irreducible quadratic factor,
  it has only one connected component.
  
   Consider the family of cubic forms
  \[
    F_c(x_0,x_1,x_2)\ =\ x_0x_2^2-\left(x_1-\frac 1c x_0\right)\bigl(x_1^2- cx_0^2\bigr)
  \]
  in one parameter $c\in\R\setminus\{0\}$.
  It is hyperbolic with respect to $(1,0,0)$ for $c>0$ and not
  hyperbolic (with respect to any point) for $c<0$. For $c=1$, it is
  hyperbolic but not strictly hyperbolic.
  \begin{center}
    \begin{minipage}{0.3\linewidth}
      \begin{center}
        \includegraphics[width=5cm]{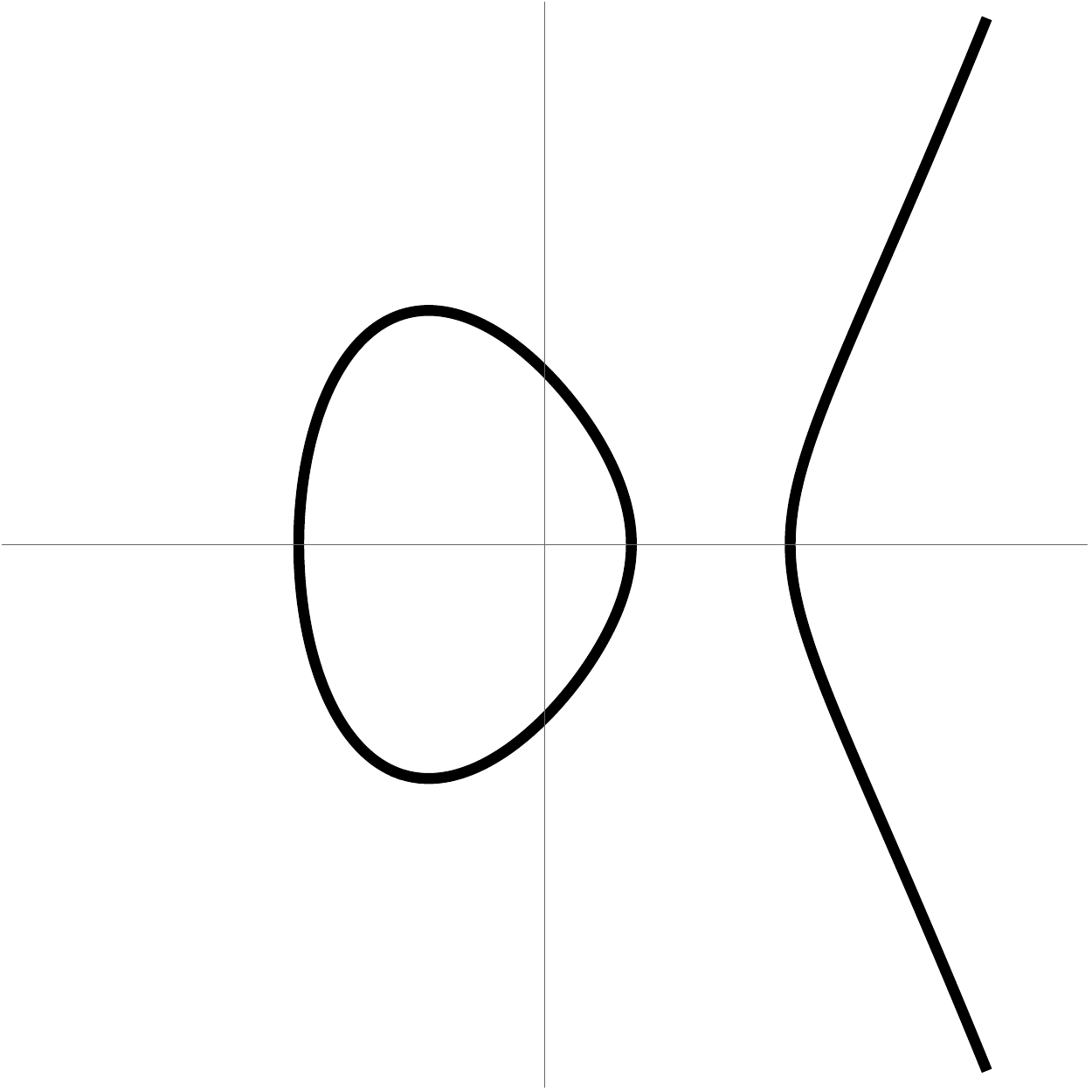}
      \end{center}
    \end{minipage}
    \begin{minipage}{0.3\linewidth}
      \begin{center}
        \includegraphics[width=5cm]{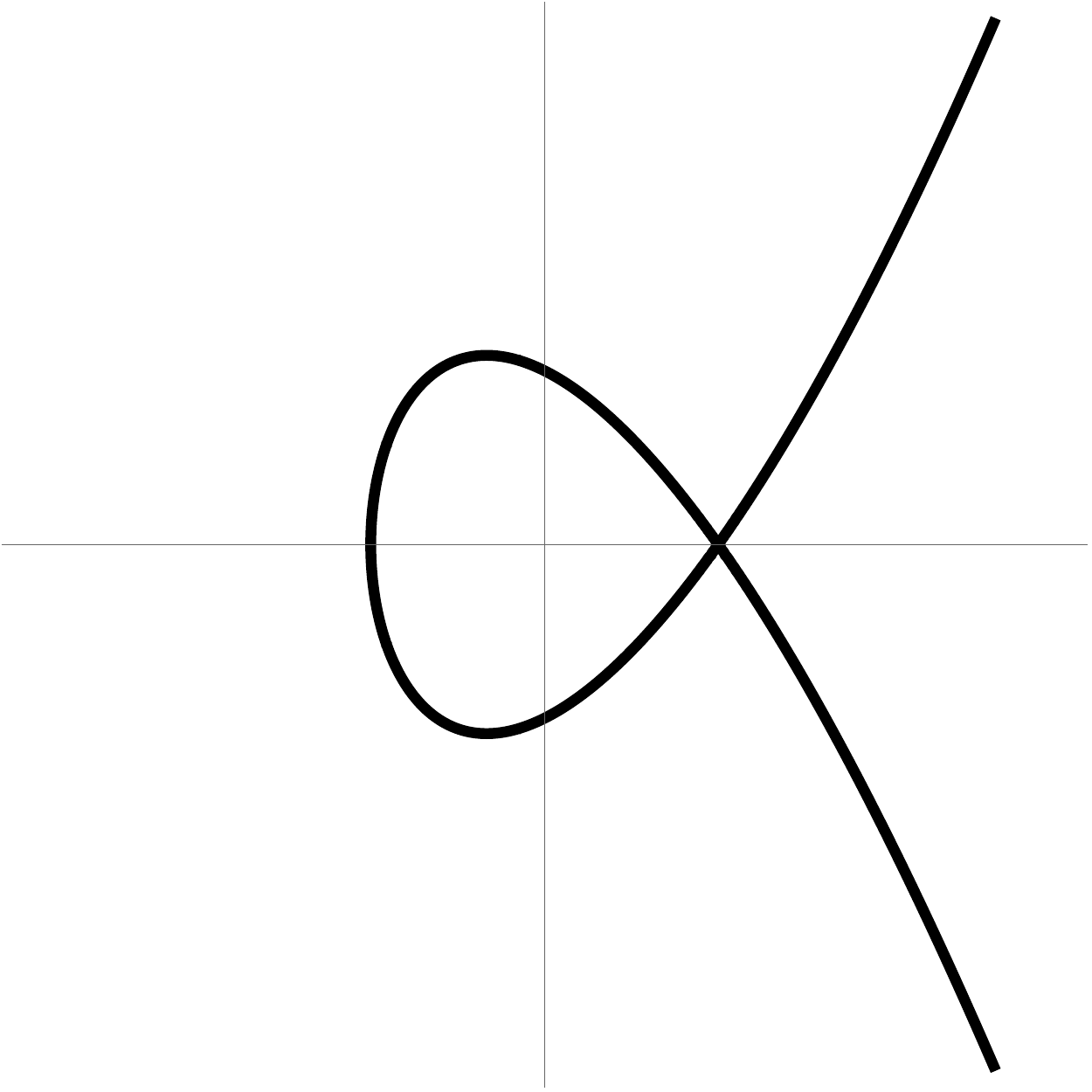}
      \end{center}
    \end{minipage}
    \begin{minipage}{0.3\linewidth}
      \begin{center}
        \includegraphics[width=5cm]{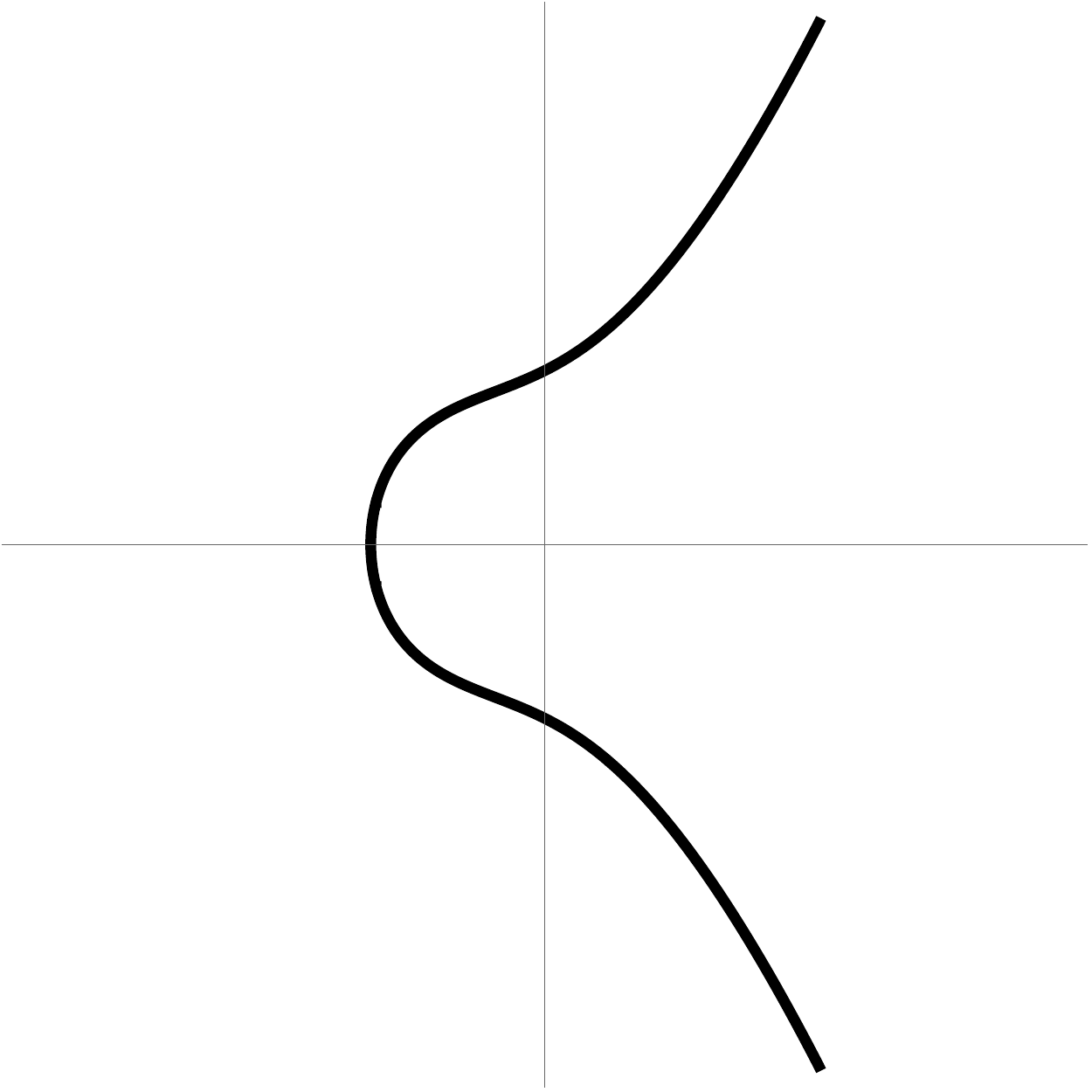}
      \end{center}
    \end{minipage}
    
    \begin{minipage}{0.7\linewidth}
      \begin{center}
        The curve defined by $F_c$ for $x_0=1$ in the
        $(x_1,x_2)$-plane for $c=2$, $c=1$ and $c=-1$.
      \end{center}
    \end{minipage}
  \end{center}
\end{Example}

In the definition of hyperbolicity, it is equivalent to ask that $F(e-ta)$ should be real rooted for all $a\in\R^{n+1}$, since $F$ is homogeneous. Moreover, it is sufficient to test real-rootedness for vectors $a\in\R^{n+1}$ orthogonal to $e$. In particular, for $e=(1,0,\dots,0)$, a form $F$ with $F(e)\neq 0$ is hyperbolic with respect to $e$ if and only if the dehomogenization $f=F(1,x_1,\dots,x_n)$ has the property that the univariate polynomial $f(ta)\in\R[t]$ is real rooted for all $a\in\R^n$. Such polynomials are called \emph{real zero polynomials}. This inhomogeneous setup is preferred in several applications.

An important variant of the definition of hyperbolicity is the following: A form $F\in\R[x_0,\dots,x_n]$ is called \emph{real stable} if it is hyperbolic with respect to every point in the positive orthant $\R^{n+1}_+$.

\section{The Hermite method}\label{sec:Hermite}

Methods to determine the number of real roots of real univariate
polynomials go back to Sturm and Hermite in the nineteenth
century. Given a monic polynomial $f\in\R[t]$ of degree $d$, the
\textit{Hermite matrix} of $f$ is the real symmetric $d\times d$-matrix
\[
  H(f)=
  \begin{pmatrix}
    N_0(f)     & N_1(f) & \cdots & N_{d-1}(f)\\
    N_1(f)     & \ddots &        & N_d(f)\\
    \vdots     &        & \ddots & \vdots\\
    N_{d-1}(f) & N_d(f) & \cdots & N_{2d-2}(f)
  \end{pmatrix}
  \]
  where $N_j(f)$ denotes the $j$-th power-sum of the complex zeros of
  $f$, which can be expressed in the coefficients of $f$ via the
  classical Newton identities. The number of distinct roots is
  given by the rank of $H(f)$ and the number of real roots by the
  signature. In particular, $f$ is real-rooted if and only
  if $H(f)$ is a positive semidefinite matrix (see \cite{Krein} for
  all of this and
  an excellent survey).

Fix $e=(1,0,\dots,0)\in\R^{n+1}$. Given a form $F[x_0,\dots,x_n]$ of
degree $d$ with $F(e)=1$, we can write down the Hermite matrix
$H_{x_0}(F)$ with respect to the variable $x_0$, whose entries are
polynomials in $x_1,\dots,x_n$. Then $F$ is hyperbolic with respect to
$e$ if and only if $H_{x_0}(F)$ is positive semidefinite for all
$a\in\R^n$. Equivalently, the \textit{Hermite form}
\[
  \mathcal{H}(F)=u^TH_{x_0}(F)u,\quad u^T=(u_1,\dots,u_d)
\]
in variables $x_1,\dots,x_n,u_1,\dots,u_d$, which is quadratic in $u$,
is a nonnegative polynomial on $\R^{n+d}$ if and only if $F$ is
hyperbolic with respect to $e$.

Nonnegativity of $\mathcal{H}(F)$ can be relaxed to a sum-of-squares
certificate. The Hermite form $\mathcal{H}(F)$ is a sum of squares in
$\R[x,u]$ if and only if the matrix $H_{x_0}(F)$ can be factored into
\[
  H_{x_0}(F)=V^TV
\]
where $V$ is a matrix with entries in $\R[x_1,\dots,x_n]$ of some
format $d\times r$. For $n\le 2$, this
relaxation is exact, but not for $n\ge 3$ (see \cite{Gondard} or
\cite{NPT13}).

\begin{Example}
  For the cubic
  \[
    F=x_0^3-\frac{x_0^2 x_1}{2}-x_0 x_1^2-\frac{x_0 x_2^2}{2}+\frac{x_1^3}{2}
  \]
 the Hermite form is given by
\begin{align*}
    \mathcal{H}(F)=&3u_{1}^2+u_{1}u_{2}x_1+\frac{9}{4}u_{2}^2x_1^2+\frac
                     92u_{1}u_{3}x_1^2+\frac 14u_{2}u_{3}x_1^3+\frac{33}{16}u_{3}^2x_1^4+
        u_{2}^2x_2^2+2u_{1}u_{3}x_2^2\\&+\frac 32u_{2}u_{3}x_1x_2^2+\frac
    52u_{3}^2x_1^2x_2^2+\frac 12u_{3}^2x_2^4
  \end{align*}
  Since $F$ is hyperbolic, this should be a sum of squares in
  $x_1,x_2,u_{1},u_{2},u_{3}$. Indeed, we computed
  \begin{align*}
    \mathcal{H}(F)&=3\left(\frac{3}{4}x_1^{2}u_{3} +\frac{1}{3}x_2^{2}u_{3}+\frac{1}{6}x_1u_{2}+u_{1}\right)^{2}+\frac{13}{6}\left(-\frac{3}{26}x_1^{2}u_{3}+ \frac{1}{26}x_2^{2}u_{3}+x_1u_{2}\right)^{2}\\
     &+\left(x_1x_2u_{3} + \frac{1}{2}x_2u_{2}\right)^{2}+\frac{3}{4}x_2^{2}u_{2}^{2}+\frac{9}{26}\left(x_1^{2}u_{3}+\frac{1}{36}x_2^{2}u_{3}\right)^{2}+\frac{47}{288}x_2^{4}u_{3}^{2}.
  \end{align*}
\end{Example}

\section{The intersection method}\label{sec:Intersection}
In this section we characterize hyperbolicity of multivariate
polynomials by separating into real and imaginary part.  For the remainder
of this section, we will use the following notation: We fix
$e=(1,0,\dots,0)\in\R^{n+1}$. Given a form $F\in\R[x_0,\dots,x_n]$, we
put
\[
f_\x(t) = F(te-\x)=F(t,\x)\text{ where }\x=(x_1,\dots,x_n)\\
\]
Recall that $F$ is hyperbolic with respect to $e$ if and only if $f_a(t)$ is real rooted for all $a\in\R^n$. For any $a\in\R^n$, the real and imaginary parts of $f_a(t)$ are polynomials in $a$, i.e.~we can write
\[
f_\x(t)=f_\real(t_1,t_2,\x)+if_\imag(t_1,t_2,\x)
\]
where $t=t_1+it_2$.

\begin{Lemma} \label{nort}
A form $F \in \R[x_0,\dots,x_n]$ with $F(e)\neq 0$ is hyperbolic with respect to $e$ if and only if the two polynomials $f_\real,f_\imag\in\R[t_1,t_2,\x]$ have no common real zero $(s_1,s_2,a)\in\R^{n+2}$ with $s_2\neq 0$.
\end{Lemma}

\begin{proof}
This is simply restating that $f_a(t)$ must be real rooted for all $a\in\R^n$.
\end{proof}

We can express the condition in the Lemma using resultants: Recall
that two non-zero polynomials $g,h\in\R[t]$ have a
commmon factor in $\R[t]$ if and only if the
resultant ${\rm Res}(g,h)=0$ in $\R$. The resultant is a polynomial in
the coefficients of $g$ and $h$. Now if we let $f=\sum_{j=0}^d
c_jt^j$ with variable coefficients and
write $t=t_1+it_2$, $f(t_1,t_2)=f_\real(t_1,t_2)+if_\imag(t_1,t_2)$,
then $f_\real$ has degree $d$ in $t_1$ while $f_\imag$ has degree
$d-1$. If $f_\real=\sum_{j=0}^d a_jt_1^j$, and
$f_\imag=\sum_{j=0}^{d-1} b_jt_1^j$ with $a_j,b_j\in\R[t_2]$, the resultant is given by the determinant of the $(2d-1)\times (2d-1)$-Sylvester matrix
\[
  {\rm Res}(f_\real,f_\imag)=
  \det
  \begin{pmatrix}
    a_0     &         &        & b_0\\
    \vdots  & \ddots  &        & \vdots & \ddots\\
    \vdots  &         & a_0    & \vdots &       & b_0\\
    a_d     &         & \vdots & b_{d-1}&       &\vdots\\
            & \ddots  & \vdots &        & \ddots&\vdots\\
            &         & a_d    &        &       & b_{d-1}
  \end{pmatrix}
\]
Since $a_d=c_d$ and $b_{d-1}=dc_dt_2$, and since $t_2$ divides the last $d$
columns of the Sylvester matrix, we have
\[
  t_2^d|{\rm Res}(f_\real,f_\imag)\quad\text{and}\quad c_d|{\rm Res}(f_\real,f_\imag).
\]
\begin{Remark}\label{Remark:FactorizationResultant}
  In fact, no higher power of $t_2$ and $c_d$ divides
  ${\rm Res}(f_\real,f_\imag)$, but it seems rather complicated to
  give a direct proof of this fact. (For example, for
  $f=c_dt^d+t^{d-1}$, one can compute
  ${\rm Res}(f_\real,f_\imag)=r\cdot c_d\cdot
  t_2^{d^2-2d+2}(1+4t_2^2c_d^2)^{d-1}$, where $r$ is a (large)
  integer, which shows that ${\rm Res}(f_\real,f_\imag)$ cannot in
  general be divisible by $c_d^2$. Similary, one can examine for
  example $f=t^d+1$ to show that $t_2$ does not in general occur to a
  higher power than $d$.
\end{Remark}

In our setup, given a form $F$ with $F(e)\neq 0$, we may assume
$F(e)=1$ so that $f_\x(t)$ is monic in $t$. We
denote by ${\rm Res}_{t_1}(f_\real,f_\imag)$ the resultant of
$f_\real,f_\imag\in\R[\x,t_1,t_2]$ with respect to the variable
$t_1$, which is quasi-homogeneous in the coefficients of $F$ (see also
\cite[Ch.~12]{Gelfand}). The relation to hyperbolicity is easy to guess, but care has to
be taken to account for possible exceptional cases. We show the following.

\begin{Thm} \label{rzres} Let $e=(1,0,\dots,0)$. Given a form $F\in\R[x_0,\dots,x_n]$ of
  degree $d$ with $F(e)=1$, the resultant of $f_\real$ and $f_\imag$ with respect to
  $t_1$ is a polynomial in $t_2$ and factors into
  \[
    {\rm Res}_{t_1}(f_\real,f_\imag)=t_2^{p}\cdot \mathcal{R}_F
  \]
  for some $p\ge d$, and $\mathcal{R}_F$ is a polynomial in $t_2,x_1,\dots,x_n$ not divisible
  by $t_2$.
  
  (1) If $F$ is hyperbolic with respect to $e$, then
  $\mathcal{R}_F$ has constant sign, i.e.~it is everywhere
  non-negative or everywhere non-positive.\\
  (2) Conversely, if
  $\mathcal{R}_F(t_2,x_1,\dots,x_n)$ does not vanish in any point
  $(s,a)\in\R\times\R^n$ with $s\neq 0$, then $F$
  is hyperbolic with respect to $e$.
\end{Thm}

\begin{proof}
  We have already observed that the resultant is always divisible by $t_2^d$.
  
  (1) Let $F$ be hyperbolic of degree $d$ and write
  $S=\{(s,a)\in\R\times\R^n\ |\ s\neq 0\}$.  Suppose for contradiction
  that $\mathcal{R}_F$ is indefinite, i.e.~there is a point $(s,a)$
  such that the sign of $\mathcal{R}_F$ is not constant in any
  neighborhood of $(s,a)$. This implies that $\mathcal{R}_F$ has an
  irreducible factor $Q$ that changes sign in $(s,a)$.  We distinguish
  two cases:

  Assume first that $(s,a)\in S$. Note that $Q$ cannot be a polynomial in
  $x_1,\dots,x_n$ alone (independent of $t_2$), since this would imply
  that $f_{a}(t)$ vanishes identically. Therefore, $Q$ must have a
  zero $(s',a')\in S$ (in fact in any neighborhood of $(s,a)$) such
  that $Q(t_2,a')$ changes sign at $t_2=s'$. Thus
  ${\rm Res}_{t_1}(f_\real,f_\imag)(t_2,a')$ has a real root at
  $t_2=s'\neq 0$ of odd multiplicity. It follows that $f_\real$ and
  $f_\imag$ have an odd number of intersection points (counted with
  multiplicity) with second coordinate $s'$. Thus there is a real such
  point, i.e.~a real number $r'$ such that $r'+is'$ is a non-real root
  of $f_a(t)$, contradicting hyperbolicity.

  If $(s,a)\notin S$, then $Q$ changes sign along 
  the hyperplane $\R^{n+1}\setminus S$, which would imply
  $Q=t_2$, contradicting $t_2\nmid\mathcal{R}_F$.
  
  (2) Suppose that $F$ is not hyperbolic. Then there is a point
  $a\in\R^n\setminus\{0\}$ for which $f_a(t)$ has a non-real
  zero. Then $f_\real(t_1,t_2,a)$ and $f_\imag(t_1,t_2,a)$ have a real
  intersection point $(t_1,t_2)=(r,s)$ with $s\neq 0$, so that $(s,a)$
  is a point in $S$ with ${\rm Res}_{t_1}(f_\real,f_\imag)(s,a)=0$ and
  hence $\mathcal{R}_F(s,a)=0$.
\end{proof}

\begin{Remark}
  It is natural to ask whether the stronger assumption in (2) is
  really needed or whether the criterion in (1) is in fact necessary
  and sufficient for hyperbolicity. We do not see how to show this
  without some further information about the factor $\mathcal{R}_F$ in
  the resultant. For instance, the converse in (1) would hold if
  $\mathcal{R}_F$ were generically irreducible or at least
  square-free. This should be expected and is verified in the examples
  below, but does not seem so easy to prove. Thus we make the
  following conjecture. If true, it would allow for a neater version
  of Thm.~\ref{rzres}.
\end{Remark}

\begin{Conjecture}
  For a generic form $F$, the factor $\mathcal{R}_F$ of the resultant in
  Theorem \ref{rzres} is irreducible.
\end{Conjecture}

For quadratic and cubic forms, the condition in Thm.~\ref{rzres} can
be made quite a bit more explicit, since hyperbolicity can be decided
by looking only at the  discriminant of $f_\x(t)$ with respect to
$t$. We will take a closer look at this case and see how it compares
to our general analysis above.

\medskip \textbf{Quadratic Forms}. Let
$F=x_0^2+f_1(\x)x_0+f_2(\x) \in \R[x_0,\x]$ be a quadratic form. It is
hyperbolic with respect to $e=(1,0,\dots,0)$ if and only if the
polynomial $f_\x(t)=t^2+f_1(\x)t+f_2(\x)$ is real-rooted for all
$\x\in\R^n$. This will
be the case if and only if $f_1^2-4f_2$ is nonnegative in $\x$
(compare also \cite{NT12} and \cite{papriquadratic}). This is a
quadratic form in $\x$, hence it is nonnegative if and only if it is a
sum of squares in $\R[\x]$.  Let us see how this translates into real
and imaginary parts, which are given by
\begin{align*}
f_\real(t_{1},t_{2})&=(t_{1}^{2}-t_{2}^{2})+t_{1}f_{1}+f_{2} \\
f_\imag(t_{1},t_{2})&=2t_{1}t_{2}+t_{2}f_{1}
\end{align*}
Thus the resultant of these two bivariate polynomials $f_\real, f_\imag$ with respect to $t_{1}$ is given by
\[
{\rm Res}_{t_1}(f_\real,f_\imag)=\det \bmatrix{1 & 2t_{2} & 0\\f_{1} & f_{1}t_{2} & 2t_{2}\\f_{2}-t_{2}^{2} & 0 & f_{1}t_{2}}=t_{2}^{2}(4f_{2}-f_{1}^{2}-4t_{2}^{2})
\]
Thus $\mathcal{R}_F$ in Theorem \ref{rzres} is
  \[
    \mathcal{R}_F=4f_2-f_1^2-4t_2^2.
  \]
  Indeed, this polynomial is nonpositive if and only $f_{1}^{2}-4f_{2}$ is
  nonnegative. 
\begin{Example}{\rm
     Let $F = x_0^2- x_{1}^{2}-x_{2}^{2} -\dots-x_{n}^{2}$. Then
\[
{\rm Res}_{t_1}(f_\real, f_\imag) = 4 t_{2}^{2}(x_{1}^{2}+x_{2}^{2}+\dots+x_{n}^{2})^{2}(1+t_{2}^{2}(x_{1}^{2}+x_{2}^{2} \dots+x_{n}^{2}))
\]
which does not vanish for $t_{2} \in \R \setminus \{0\}$ and any
$\x \in \R^{n}$. Hence $F$ is hyperbolic. Also note that
$f^{2}_{1}-4f_{2}=4(x_{1}^{2}+\dots+x_{n}^{2})$ is a sum of
squares.  }
\end{Example}

\bigskip
\noindent\textbf{Cubic Forms}.
Let
\[
  F(\x)=x_0^3+f_1(\x)x_0^2+f_2(\x)x_0+f_3(\x)
\]
be a cubic form. Again, hyperbolicity of $F$ with respect to
$e=(1,0,\dots,0)$ is equivalent to
$t^3+f_1(\x)t^{2}+f_2(\x)t+f_3(\x)$ being real rooted in $t$ for
all $\x=a\in\R^n$. This is the case if and only if the cubic
discriminant $\Delta$ of $f_\x(t)$ is nonnegative for all $\x$. It is given by
\[
  \Delta=18f_{1}f_{2}f_{3}-4f_{2}^{3}+f_{1}^{2}f_{2}^{2}-4f_{1}^{3}f_{3}-27f_{3}^{2}
\]
Again, we compare this to our resultant. The real and imaginary parts
are given by
\begin{align*}
f_\real(t_{1},t_{2})&=t_{1}^{3}+t_{1}^{2}f_{1}+t_{1}(f_{2}-3t_{2}^{2})-t_{2}^{2}f_{1}+f_3 \\
f_\imag(t_{1},t_{2})&=3t_{1}^{2}t_{2}+2t_{1}t_{2}f_{1}+t_{2}f_{2}-t_{2}^{3}.
\end{align*}
Thus the resultant of $f_\real$ and $f_\imag$ with respect to $t_{1}$
comes out as
\begin{align*} 
{\rm Res}_{t_1}(f_\real,f_\imag)
&=\bmatrix{1 & 0 & 3t_{2} & 0& 0\\f_{1} &1& 2t_{2}f_{1}&3t_{2}& 0\\f_{2}-3t_{2}^2& f_{1}& t_{2}f_{2}-t_{2}^3& 2t_{2}f_{1} &3t_{2}\\f_{3}-t_{2}^2f_{1}& f_{2}-3t_{2}^2& 0 &t_{2}f_{2}-t_{2}^3& 2f_{1}t_{2}\\0 &f_{3}-t_{2}^2f_{1}& 0& 0& t_{2}f_{2}-t_{2}^3}\\
  &=-t_{2}^{3}[\Delta+t_{2}^{2}g_{2}+t_{2}^{4}g_{3}+t_{2}^{6}g_{4}].
\end{align*}
where
\[
  g_{2}=4(f_{1}^{2}-3f_{2})^2,\quad
  g_{3}=32(f_{1}^{2}-3f_{2}),\quad g_{4}=64.
\]
Thus $\mathcal{R}_F$ in Thm.~\ref{rzres} is the polynomial
\[
  \mathcal{R}_F=\Delta+t_{2}^{2}g_{2}+t_{2}^{4}g_{3}+t_{2}^{6}g_{4}.
\]
In this case, we find indeed that nonnegativity of $\Delta$ is
equivalent to nonnegativity of $\mathcal{R}_F$. To see this, note that
$\mathcal{R}_F$ is nonnegative if and only if the cubic equation
  \begin{equation}
    \Delta+t_{2}^{2}g_{2}+t_{2}^{4}g_{3}+t_{2}^{6}g_{4}=0\label{eq:cubicdiscriminant}
  \end{equation} in $t_2^2$ has no
  positive real root. Assume that $\Delta$ is nonnegative. This means
  that
  $4(f_{1}^{2}-3f_{2})^{3}-(2f_{1}^{3}-9f_{1}f_{2}+27f_{3})^{2}$
  is nonnegative. Therefore, $f_{1}^{2}-3f_{2}$ and hence $g_3$ must be
  nonnegative. This shows that
  \eqref{eq:cubicdiscriminant} has only nonnegative coefficients and
  hence  no positive solution in
  $t_2^2$. Thus $\mathcal{R}_F$ is nonnegative.

\bigskip
For forms of degree at least $4$, it is not enough to consider only
the discriminant, as the following simple example shows.
\begin{Example}
  The quartic
  \[
    F=x_0^4-x_1^4-x_2^4
  \]
  is not hyperbolic with respect to $e=(1,0,0)$. However,
  $f_\x(t)=t^4-(x_1^4+x_2^4)$ has distinct roots in $t$ for all
  $(x_1,x_2)\neq (0,0)$, hence the discriminant of $f_\x(t)$ has constant
  sign.

  That $F$ is not hyperbolic is however reflected in the fact that
  \[
    \mathcal{R}_F=256\bigl(t_2^4-x_1^4-x_2^4\bigr)\bigl(4t_2^4+x_1^4+x_2^4\bigr)^2
  \]
  is clearly neither nonnegative nor nonpositive.
\end{Example}

\begin{Example}
Consider our parametrized cubic
\begin{equation*}
    F_c(x_0,x_1,x_2)\ =\ x_0x_2^2-\left(x_1-\frac 1c x_0\right)\bigl(x_1^2- cx_0^2\bigr)
\end{equation*}
As noted, it is hyperbolic with respect to $(1,0,0)$ if and only if $c >0$. Indeed, substituting $c=b^2$, 
we can represent the discriminant $\Delta$ as the sum of squares
\begin{equation*}
 \Delta=(2b^6-2)^2x^6+x^2y^4b^{10}+20x^4y^2b^8+4y^6b^6+12x^2y^4b^4+12x^4y^2b^2.
\end{equation*}
\end{Example}

\subsection{Using the real Nullstellensatz}
Since the conditions in Theorem \ref{rzres} need to be satisfied for
any $\x \in \R^{n}$ and the resultants quickly become quite large, it
is not clear how useful this method is in practice. Of course, it is
not necessary to rely on resultants to test whether the real and
imaginary part of a polynomial intersect. One can also employ the
real Nullstellensatz, which will also translate into a sums-of-squares
condition.

The real Nullstellensatz is the following general criterion for
infeasibility; see \cite{BCR},\cite{Ma08}.
\begin{Thm}[Real Nullstellensatz]  A system $f_1,\dots, f_k\in \R[\x]$
  of real polynomials in variables $\x=(x_1,\dots,x_n)$ has no common
  zero in $\R^n$ if and only if there exist polynomials $q_1,\dots,
  q_k$ and a sum of squares $s$ in $\R[\x]$ such that
  \[
    s + q_1f_1 + \dots + q_kf_k = -1
  \]
\end{Thm}

Reading the identity in the real Nullstellensatz modulo the ideal
generated by $f_1,\dots,f_k$, we obtain the following equivalent
formulation: If $\I$ is an ideal $\R[\x]$, then the real variety
$\V_\R(\I)$ defined by $\I$ in $\R^n$ is empty if and only if $-1$ is
a sum of squares in the residue ring $\R[x]/\I$.

Testing this sum-of-squares condition can be translated into a
semidefinite program, either directly or 
combined with a Gr\"obner basis computation working in $\R/\I$. We did
some experiments in \texttt{Macaulay2} with the \texttt{SOSm2} package
(\cite{Macaulay2}, \cite{Parrilososcomp}).

Applying this to our problem, we are given a form $F\in\R[x_0,\x]$ and
wish to test for hyperbolicity with respect to $e=(1,0,\dots,0)$. We
form $f_\x(t)$ and decompose into real and imaginary part. Then $F$ is
hyperbolic if and only if $f_\real$ and $f_\imag$ have no common real
root in $t_1,t_2,\x$ with $t_2\neq 0$ (Lemma \ref{nort}). This is
equivalent to the system
\[
 f_\real,\ f_\imag,\ 1-yt_2
\]
with one additional variable $y$ being infeasible. Thus we obtain the
following criterion for hyperbolicity.

\begin{Prop} \label{proprzsdp}
Let $F\in\R[x_0,\dots,x_n]$ be a form of degree $d$ with $F(e)\neq 0$
and let $\I$ be the ideal generated by $f_\real$, $f_\imag$, $1-yt_2$
in $A=\R[t_1,t_2,\x,y]$. Then $F$ is hyperbolic with
respect to $e$ if and only if $-1$ is a sum of squares in $A/\I$. \qed
\end{Prop} 

\section{The discriminant method}\label{sec:Discriminant}
Our final method for testing hyperbolicity is based on an observation due to Nuij in \cite{Nui68}, also used in \cite{PL17}. We will work in the following setup.
Let $e=(1,0,\dots,0)\in\R^{n+1}$, $d\ge 1$, $\x=(x_1,\dots,x_n)$ as
before, and consider the sets
\begin{align*}
\mathcal{F}&=\bigl\{F\in\R[t,\x]\: |\: F\text{ is homogeneous of degree }d\text{ and }F(e)=1\bigr\}\\
\mathcal{H}&=\bigl\{F\in\mathcal{F}\:|\: F\text{ is hyperbolic with respect to }e\}.
\end{align*}
Note that $F$ lies in $\mathcal{H}$ if and only if $F(t,a)$ is real
rooted in $t$ for all $a\in\R^n$ (c.f.~\S\ref{sec:Hyperbolicity}). 

Nuij constructed an explicit path in the space of polynomials
connecting any given polynomial to a fixed polynomial in
$\mathcal{H}$. We consider the following operators on polynomials $\mathcal{F} \subset \R[t,\x]=\R[t,x_1,\dots,x_n]$. 
\begin{align*}
  T^\ell_s &\colon F\mapsto F+s \ell\frac{\partial F}{\partial
    t}\quad (\ell\in\R[\x] \text{ a linear form})\\
  G_s &\colon F\mapsto F(t,s\x)\\
  H_s &= (T^{x_1}_s)^d\cdots (T^{x_n}_s)^d\\
  N_s &= H_{1-s}G_s\,,
\end{align*}
where $s\in\R$ is a parameter. For fixed $s$, all of these
are linear operators on $\R[t,\x]$ taking the affine-linear subspace
$\mathcal{F}$ to itself. Clearly, $G_s$ preserves hyperbolicity for any $s\in\R$,
and $G_0(f)=t^d$ for all $F\in\mathcal{F}$. The operator $H_s$ is used to
smoothen the polynomials along the path $s\mapsto G_s(F)$. The
exact statement is the following.

\begin{Prop}[Nuij \cite{Nui68}] \label{Prop:Nuij}
For $s\ge 0$, the operators $T^\ell_s$
  preserve hyperbolicity. Moreover, the following holds:
  \begin{enumerate}
	  \item For any $F\in\mathcal{F}$, we have $N_1(F)=F$. 
	  \item The polynomial $N_0(F)$ lies in ${\rm int}(\mathcal{H})$ and is independent of $F$. 
	  \item For $F\in\mathcal{H}$, we have $N_s(F)\in{\rm int}(\mathcal{H})$ for all $s\in [0,1)$.\qed
  \end{enumerate}
\end{Prop}

For $F\in\mathcal{F}$, we call $[0,1]\ni s\mapsto N_s(F)$ the \emph{N-path} of $F$. 

\begin{Cor}\label{Cor:NpathCross}
	A form $F\in\mathcal{F}$ is hyperbolic if and only if the N-path does not cross the boundary of $\mathcal{H}$, i.e.~$N_s(F)\in\mathcal{H}$ for all $s\in (0,1)$.\qed
\end{Cor}

The boundary of $\mathcal{H}$ is a subset of the hypersurface in
$\mathcal{F}$ defined by the vanishing of the discriminant
$\Delta\in\R[\x]$ of polynomials in $\mathcal{F}$ with
respect to the variable $t$. The polynomial $\Delta$ can be expressed
via the Sylvester matrix and is homogeneous of degree $2d-2$. We can
test for the criterion in Cor.~\ref{Cor:NpathCross} by restricting the discriminant to the N-path, as follows: Let $F\in\mathcal{F}$ and write
	\[
	\Delta_{N}(F)=\Delta(N_s(F))\in\R[s,\x].
	\]
We call $\Delta_{N}$ the \emph{N-path discriminant}. Our preceding discussion translates to the following statement:

\begin{Cor}\label{Cor:NpathDisc}
	Let $F\in\mathcal{F}$. If $\Delta_{N}(F)(s,a)\neq 0$ for all $s\in (0,1)$ and $a\in\R^n$, then $F$ is hyperbolic.\qed
\end{Cor}

The converse is not quite true, but we have the following
characterization of hyperbolicity, which is analagous to what we found
for the intersection method.

\begin{Thm}\label{Thm:NpathHyp}
  If a polynomial $F\in\mathcal{F}$ is hyperbolic, then 
   \[
     \Delta_{N}(F)(s,a)\ge 0
   \]
   holds for all $s\in [0,1]$ and $a\in\R^n$. Conversely, if
   $\Delta_N(f)(s,a)>0$ holds for all $s\in [0,1]$ and
   $a\in\R^n\setminus\{0\}$, then $F$ is strictly hyperbolic. 
 \end{Thm}

 \begin{proof}
   Suppose first that $F$ is hyperbolic. By continuity, we may assume
   $F\in{\rm int}(\mathcal{H})$, which means that $F$ is strictly
   hyperbolic. It follows that $N_s(F)(t,a)$ has distinct real roots
   in $t$ for all $a\in\R^n\setminus\{0\}$, hence $\Delta_N(F)(s,a)>0$
   for all $s\in [0,1]$. Since $\Delta_N(F)(0,a)>0$ for all
   $a\in\R^n\setminus\{0\}$, we conclude that  $\Delta_N(F)(s,a)\ge 0$ holds for all
   $s\in [0,1]$.

   If $F$ is not hyperbolic, then $F(t,a)$ has a non-real root for
   some $a\in\R^n\setminus\{0\}$. Since $N_0(F)(t,a)$ has distinct
   real roots, it follows that $\Delta_N(F)(s,a)$
   must vanish for some $s\in (0,1]$.
 \end{proof}

In the case of curves ($n=2$), the positivity condition on the hyperbolicity discriminant can be related to a beautiful result due to Marshall:

\begin{Thm}[\cite{Ma10}]\label{Thm:Marshall} A polynomial $h\in\R[s,t]$ satisfies $h(a,b)\ge 0$ for all $a\in [0,1]$ and $b\in\R$ if and only if there exist sums of squares $\sigma_1,\sigma_2\in\R[s,t]$ such that
	\[
	h = \sigma_1+\sigma_2\cdot s(1-s).
	\]
\end{Thm}

The proof of Marshall's theorem is quite intricate. Unfortunately, the degree of the sums of squares $\sigma_1$ and $\sigma_2$ cannot be bounded in terms of the degree of $h$ alone. Therefore, Thm.~\ref{Thm:Marshall} does not translate into a criterion that can be checked by a single semidefinite program. Nevertheless, an SDP hierarchy of growing degrees can be employed.
The analogue of Thm.~\ref{Thm:Marshall} does not hold if more than one variable is unbounded. We refer to Marshall's book \cite{Ma08} for a broader discussion.  

\begin{Cor}\label{Cor:NpathMarshall} If a form $F\in\R[t,x_1,x_2]$ is hyperbolic, then there exist sums of squares $\sigma_1,\sigma_2\in\R[s,y]$ such that 
\[
\Delta_{N}(F)(s,y,1)=\sigma_1+\sigma_2 s(1-s).
\]
\end{Cor}

\begin{proof}
  The polynomial $\Delta_{N}(F)(s,x_1,x_2)$ is homogeneous in
  $x_1,x_2$, hence, if it is non-negative for $s\in [0,1]$, then so is
  the dehomogenization $\Delta_N(F)(s,y,1)$. 
\end{proof}

\begin{Example}
  The hyperbolic cubic
  \[
        F=t^3-\frac{t^2 x_1}{2}-t x_1^2-\frac{t x_2^2}{2}+\frac{x_1^3}{2}
  \]
  has the N-path discriminant\\[-1.5em]

  {\scriptsize  \begin{align*}
             \Delta_{N}(F)\ =\ &
\frac{29469 s^6 x_1^6}{4}+\frac{51283}{2} s^6 x_1^4 x_2^2-3316 s^6 x_1^3
                             x_2^3+\frac{392497}{16} s^6 x_1^2 x_2^4-36 s^6
                                 x_1 x_2^5+\frac{12169 s^6 x_2^6}{2}\\
                    &-39350 s^5
                             x_1^6-143390 s^5 x_1^4 x_2^2+20316 s^5 x_1^3
               x_2^3-139200 s^5 x_1^2 x_2^4+108 s^5 x_1
                      x_2^5-34632 s^5 x_2^6\\
                    &+89581 s^4 x_1^6+338905
                             s^4 x_1^4 x_2^2-51420 s^4 x_1^3 x_2^3+332832 s^4
                             x_1^2 x_2^4-108 s^4 x_1 x_2^5+82980 s^4
                      x_2^6\\
                    &-111308 s^3 x_1^6-433116 s^3 x_1^4
                             x_2^2+68980 s^3 x_1^3 x_2^3-429120 s^3 x_1^2
                      x_2^4+36 s^3 x_1 x_2^5-107136 s^3 x_2^6\\
                    &+79632 s^2
                             x_1^6+315648 s^2 x_1^4 x_2^2-51840 s^2 x_1^3
                             x_2^3+314640 s^2 x_1^2 x_2^4+78624 s^2
                      x_2^6-31104 s x_1^6\\
                    &-124416 s x_1^4 x_2^2+20736 s
                             x_1^3 x_2^3-124416 s x_1^2 x_2^4-31104 s x_2^6+5184
                      x_1^6+20736 x_1^4 x_2^2\\
                    &-3456 x_1^3 x_2^3+20736 x_1^2
                             x_2^4+5184 x_2^6.
           \end{align*}
         }

         \vspace*{-1em}
         \noindent It can be verified numerically that $\Delta_N(F)$ is indeed
         nonnegative for $0\le s\le 1$, but we could not derive a nice
         rational representation of the form in Cor.~\ref{Cor:NpathMarshall}.
\end{Example}

\end{document}